\documentclass[11pt]{article}
\usepackage{dsfont,amsmath,amssymb,amsthm}

\renewcommand{\aa}{\mathbf{a}}
\newcommand{\bb}{\mathbf{b}}
\newcommand{\cc}{\mathbf{c}}
\newcommand{\dd}{\mathbf{d}}
\newcommand{\ee}{\mathbf{e}}
\newcommand{\uu}{\mathbf{u}}
\newcommand{\vv}{\mathbf{v}}

\newcommand{\ww}{\mathbf{w}}
\newcommand{\xx}{\mathbf{x}}
\newcommand{\yy}{\mathbf{y}}
\newcommand{\R}{\mathds{R}}
\newcommand{\PP}{{\cal P}}

\newtheorem{definition}{Definition}
\newtheorem{lemma}[definition]{Lemma}
\newtheorem{theorem}[definition]{Theorem}
\newtheorem{fact}[definition]{Fact}

\title{Large Shadows from Sparse Inequalities\thanks{The project CG Learning acknowledges the financial support of the Future and Emerging Technologies (FET) programme within the Seventh Framework Programme for Research of the European Commission, under FET-Open grant number: 255827}}
\author{Bernd G\"artner \qquad Christian Helbling \\Institute of Theoretical Computer Science \\
  ETH Zurich \\ Switzerland \and Yoshiki Ota \qquad
  Takeru Takahashi \\ Graduate School of Information Sciences \\
  Tohoku University \\ Japan}

\begin{document}

\maketitle

\begin{abstract}
  The $d$-dimensional Goldfarb cube is a polytope with the property
  that all its $2^d$ vertices appear on some \emph{shadow} of it
  (projection onto a $2$-dimensional plane). The Goldfarb cube is the
  solution set of a system of $2d$ linear inequalities with at most
  $3$ variables per inequality. We show in this paper that the
  $d$-dimensional Klee-Minty cube --- constructed from inequalities
  with at most $2$ variables per inequality --- also has a shadow with
  $2^d$ vertices. In contrast, with one variable per inequality, the
  size of the shadow is bounded by $2d$.
\end{abstract}

\section{Introduction}
The study of shadows of polytopes goes back to 1955, when Gass and
Saaty introduced a variant of the simplex method
for solving linear programs whose objective function linearly depends
on a real parameter $\lambda$~\cite{Gass:1955hg}. For every fixed
value of $\lambda$, the problem can be treated as an ordinary linear
program, but the approach of Gass and Saaty was to compute the optimal
value as an explicit (piecewise linear) function of $\lambda$, and
afterwards simply look up the solution for any desired parameter value
$\lambda$.  50 years later, this approach was rediscovered in the
machine learning community, in the context of support vector machines
(parameterized quadratic programs)~\cite{Hastie:2004uj}.

\paragraph{The Gass-Saaty method.}
Let us assume for the discussion here that the feasible region of the
linear program is a simple polytope $\PP\subseteq\R^d$ with $n$ facets
(for background on polytopes and this geometric view of linear
programming, we refer to Ziegler's book~\cite[Section
3.2]{z-lop-95}). We also assume that the objective function is of the
form $f_{\lambda}(\xx)=\cc^T\xx + \lambda\dd^T\xx$, where
$\cc,\dd\in\R^d$ are linearly independent and generic (non-constant on
every edge of $P$). Then the output of the Gass-Saaty method is a
sequence of vertices $\vv_0,\vv_1,\ldots, \vv_{M-1}$ of $\PP$, along
with a sequence of real values
$-\infty=\lambda_0<\lambda_1<\cdots<\lambda_M=\infty$, with the
following property:
\begin{quotation}
  $\forall k\in\{0,1,\ldots,M\}$: $\vv_k$ maximizes
  $f_{\lambda}$ over $\PP$ for $\lambda\in[\lambda_k,\lambda_{k+1}]$.
\end{quotation}
Hence, for $\lambda\in[\lambda_k,\lambda_{k+1}]$, the optimal value of
the linear program is $\cc^T\vv_k+\lambda \dd^T\vv_k$, so we indeed
get the optimal objective function value as a pieceweise linear
function in $\lambda$, with $M$ ``bends''.

The two sequences are computed as follows: by solving an ordinary
linear program, we initally find the vertex $\vv_0$ that maximizes
$-\dd^T\xx$, corresponding to parameter value $\lambda_0=-\infty$. Now
suppose that we have already computed $\vv_k$ and a value of
$\lambda_k$ for which $\vv_k$ is optimal. Starting from $\lambda_k$,
we grow $\lambda$ until we have a (unique) neighboring vertex
$v_{k+1}$ with $f_{\lambda}(\vv_k)=f_{\lambda}(\vv_{k+1})$. The
corresponding value of $\lambda$ will be $\lambda_{k+1}$. If $\lambda$
can grow indefinitely without reaching the former equality, we have
$k={M-1}$ and set $\lambda_{M}=\infty$. Algebraically, the
computations are very simple, if the \emph{simplex method} is used. At
value $\lambda_k$, we have a certificate of optimality of $\vv_k$, in
the form of nonpositive \emph{reduced costs} that are also linear
functions in $\lambda$. Hence we can compute $\lambda_{k+1}$ as the
next higher value for which some reduced cost coefficient is about to
become positive. At this point, a single \emph{pivot step} will yield
$\vv_{k+1}$. For details, we refer to the original article by Gass and
Saaty~\cite{Gass:1955hg}.

\paragraph{The shadow vertex method.}
The Gass-Saaty method can also be used to solve an ordinary linear
program with objective function $\cc^T\xx$, given some initial vertex
$\vv_0$. For this, we compute an auxiliary objective function $\dd$
which is uniquely minimized by $\vv_0$ (this is easy); then we run the
Gass-Saaty method until we get to an interval
$[\lambda_k,\lambda_{k+1}]$ containing $0$. The corresponding vertex
$\vv_k$ maximizes $\cc^T\xx$ over $\PP$.

\paragraph{Shadows of polytopes.}
It is clear that the efficiency of the Gass-Saaty method critically
depends on the number of bends $M$, equivalently the number of
vertices in the sequence $\vv_0,\vv_1,\ldots, \vv_{M-1}$. Each of them
turns out to be a \emph{shadow vertex}, meaning that we still ``see
it'' when we project the polytope $P$ to the 2-dimensional plane
spanned by the vectors $\cc$ and $\dd$ (see Lemma~\ref{lem:shadow} below).

In many cases, the number of shadow vertices is small. For example,
when we project the unit cube $[0,1]^d$ to any 2-dimensional plane, we
obtain a polygon with at most $2d$ vertices
(see Theorem~\ref{thm:zonotope} below). In the worst case, however, shadows can
be large. Murty (in the dual setting) was the first to construct
shadows whose size is exponential in the dimension of the
polytope~\cite{Murty1980}. A more explicit primal construction was
later provided by Goldfarb in form of a $d$-dimensional defomed cube,
with all its $2^d$ vertices appearing on some 2-dimensional
shadow~\cite{Goldfarb1,Goldfarb2}. A further simplification of the
construction is due to Amenta and Ziegler~\cite[Section 4.3]{AZ}). The
Goldfarb cube also serves as the starting point for an exponential
lower bound on the complexity of a support vector machine's 
\emph{regularization path}~\cite{gjm-elbpcrp-12}.

\paragraph{The effect of sparsity.} Let us now consider an explicit
inequality description the polytope $\PP$, 
\[
\PP =\{\xx\in\R^d: A\xx \leq \bb\}, \quad A\in\R^{n\times d}, b\in\R^n.
\]
We call this description \emph{$t$-sparse} if
every row of the matrix $A$ has at most $t$ nonzero entries. In other
words, if every inequality contains at most $t$ variables. The
question that we ask in this paper is the following: what is the
effect of sparsity on the size of the 2-dimensional shadows of $\PP$?

The Goldfarb cube in the version of Amenta and Ziegler~\cite[Section
4.3]{AZ}) has a $3$-sparse inequality description, meaning that
exponentially large shadows require at most $3$ variables per
inequality. On the other hand, any polytope with a $1$-sparse
inequality description is an axis-parallel box, with 2-dimensional
shadows of size $2d$ at most (Theorem~\ref{thm:zonotope}).

This means, the interesting case is the $2$-sparse one. We were
initially hoping that $2$-sparsity entails small shadows as well.  In
this paper, we show that this is not the case, by constructing a
$d$-dimensional polytope with a $2$-sparse description by $2d$
inequalities, having a $2$-dimensional shadow of size $2^d$. In fact,
this polytope is the well-known \emph{Klee-Minty cube}~\cite{KM}, with
carefully chosen projection vectors $\cc$ and $\dd$.

In the next section, we formally define shadows of polytopes and prove
some basic properties. Section~\ref{sec:1-sparse} deals with the
$1$-sparse case. The fact that $2$-dimensional shadows are small in
this case is well-known; we will provide a simple proof for the sake
of completeness. Section~\ref{sec:2-sparse} contains the main
contribution of this paper: a 2-dimensional shadow of the
$d$-dimensional Klee-Minty cube with $2^d$ vertices.

\section{Projections and Shadows}
Let $\cc,\dd\in\R^d$ be linearly independent vectors. We consider the
\emph{linear projection} $\pi_{\cc,\dd}:\R^d\rightarrow\R^2$ defined
by
\begin{equation}
\label{eq:proj}
\pi_{\cc,\dd}(\xx) = \left(\begin{array}{ll}\cc^T\xx\\ \dd^T\xx\end{array}\right).
\end{equation}

\begin{definition}
  For a point set ${\cal P}\subseteq\R^d$, the \emph{$2$-dimensional
    shadow} (or simply \emph{shadow}) of $\PP$ w.r.t.\ $\cc$ and $\dd$
  is
\[\pi_{\cc,\dd}({\cal P}) := \{\pi_{\cc,\dd}(\xx): \xx\in{\cal P}\}.\]
\end{definition}
If ${\cal P}$ is a polytope---the convex hull of its
vertices~\cite[Proposition 2.2~(i)]{z-lop-95}---then its shadow is
easily seen to be a polytope as well: the shadow is the convex hull of
the projected vertices, some of which are the actual vertices of the
shadow~\cite[Proposition 2.2~(ii)]{z-lop-95}. Hence we have the
following fact.

\begin{fact}\label{fact:shadow} 
Let ${\cal P}$ be a polytope, and $\ww$ be a vertex of the shadow
$\pi_{\cc,\dd}({\cal P})$. Then there exists a vertex $\vv$ of ${\cal P}$ such that
$\ww=\pi_{\cc,\dd}(\vv)$.
\end{fact}

The next lemma provides a sufficient condition for a vertex to
actually yield a shadow vertex.
\begin{lemma}\label{lem:shadow}
Let ${\cal P}$ be a polytope, and let $\vv$ be a vertex of ${\cal P}$. The
projection $\ww=\pi_{\cc,\dd}(\vv)$ is a vertex of the shadow 
$\pi_{\cc,\dd}({\cal P})$ if there exists a linear combination 
$\ee$ of $\cc$ and $\dd$ such that
$\vv$ is the unique maximizer of the linear function
$\ee^T\xx$ over ${\cal P}$.
\end{lemma}
\begin{proof}
For $\aa=(a_1,a_2)\in\R^2$, define $\ee=a_1\cc+a_2\dd$. With
$\yy=\pi_{\cc,\dd}(\xx)$,
we have 
\begin{equation}
\label{eq:ea}
\ee^T\xx = a_1\cc^T\xx + a_2\dd^T\xx = \aa^T\yy.
\end{equation}
Thus, if vertex $\vv$ is the unique maximizer of $\ee^T\xx$ over
${\cal P}$, then $\ww=\pi_{\cc,\dd}(\vv)$ is the unique maximizer of
$\aa^T\yy$ over $\pi_{\cc,\dd}({\cal P})$. This in turn means that 
$\ww$ is a vertex of the shadow~\cite[Definition
2.1]{z-lop-95}.
\end{proof}

\section{The $1$-Sparse Case}\label{sec:1-sparse}
Let us consider a system of inequalities in $d$ variables $x_1,x_2,\ldots,x_d$, such that each inequality contains only one variable. Hence, the inequality is either an upper or a lower bound for that variable. By considering the tightest lower
and upper bounds for each variable, we see that the set of solutions consists of 
all $\xx\in\R^d$ such that 
\begin{equation}\label{eq:box}
\ell_i \leq x_i \leq u_i, \quad i=1,2,\ldots,d,
\end{equation}
for suitable numbers $\ell_i<u_i$ (we assume that the solution set is
full-dimensonal and bounded). The vertices of the polytope ${\cal
  P}$---a box---defined by these inequalities are therefore all the
$2^d$ points $\xx$ for which $x_i\in\{\ell_i,u_i\}$ for all $i$.

\begin{theorem}\label{thm:zonotope} Let ${\cal P}$ be a box as in (\ref{eq:box}). Then 
$\pi_{\cc,\dd}({\cal P})$ has at most $2d$ vertices.
\end{theorem}

\begin{proof} We may assume that there is no $i$ such that $c_i=d_i=0$
  (otherwise we reduce the dimension of the problem by ignoring
  coordinate $i$ and obtain a bound of $2(d-1)$). We now prove that
  there are at most $2d$ vertices of ${\cal P}$ that project to some
  vertex $\ww$ of $\pi_{\cc,\dd}({\cal P})$. We recall that $\ww$ is a
  vertex if and only if $\ww$ is the unique maximizer of $\aa^T\yy$
  over $\pi_{\cc,\dd}({\cal P})$ for suitable $\aa\in\R^2$~\cite[Definition
  2.1]{z-lop-95}. Again, we set $\ee=a_1\cc+a_2\dd$.

  Since $\ww$ is the \emph{unique} maximizer, we can slightly perturb
  $\aa$ and w.l.o.g.\ assume that $e_i=a_1c_i+a_2d_i\neq 0$ for all
  $i$. We now claim that the sign pattern of $\ee$ uniquely determines
  the preimage $\vv$ of $\ww$. To see this, we use (\ref{eq:ea}) to
  argue that any preimage of $\ww$ maximizes $\ee^T\xx$ over ${\cal
    P}$.  But under $e_i\neq 0$ for all $i$, there is only one such
  maximizer $\vv$, given by:
  \[
  v_i = \left\{\begin{array}{ll}
        \ell_i, &\mbox{if $e_i<0$} \\
        u_i & \mbox{if $e_i>0$}.\
      \end{array}\right., i=1,2,\ldots,d.
      \]

Thus, the theorem follows if we can prove that there are at most $2d$ different
sign patterns that may occur in $\ee$. For each $i$, we consider the
line
\[L_i = \{\yy\in\R^2: y_1 c_i + y_2 d_i = 0\}\] through the
origin. The \emph{arrangement} of all $d$ such lines subdivides the
plane into \emph{cells} where all points $\aa$ within a fixed cell
lead to the same sign pattern of $\ee$. The twodimensional cells
correspond to the nowhere zero sign patterns of interest. It remains
to observe that an arrangement of $d$ lines through the origin induces
at most $2d$ twodimensional cells.
\end{proof}

We remark that we have reproved a special case of a general statement
that relates \emph{zonotopes} to \emph{arrangements of
  hyperplanes}~\cite[Corollary 7.17]{z-lop-95}.

\section{The $2$-Sparse Case}\label{sec:2-sparse}
We begin by introducing the $d$-dimensional Klee-Minty cube in the
variant of Amenta and Ziegler~\cite[Section 4.1]{AZ}. The original
Klee-Minty cube---the first and celebrated worst-case input for the
simplex method (with Dantzig's pivot rule)---differs from this variant
by a suitable scaling of the inequalities~\cite{KM}.

\begin{definition} For fixed $0<\varepsilon<1/2$, the $d$-dimensional
  Klee-Minty cube is the set of solutions of the following system of
  $2d$ inequalities that come in $d$ pairs, where the $j$th pair specifies
  a lower and an upper bound for variable~$x_j$.
\begin{equation}\label{eq:km}
\begin{array}{rcccl}
0 &\leq& x_1 &\leq& 1 \\
\varepsilon x_{j-1} &\leq& x_j &\leq& 1- \varepsilon x_{j-1}, \quad j=2,\ldots,d. 
\end{array}
\end{equation}
\end{definition}

\subsection{Vertices}
It is easily shown by induction that $0\leq x_j\leq 1$ for every
$\xx=(x_1,x_2,\ldots,x_d)$ in the polyhedron (\ref{eq:km})
and every $j$. Hence, we are dealing with a polytope.
Using $\varepsilon<1/2$, this in turn implies that from
any pair of inequalities, at most one can be tight. A vertex of the
polytope (having $d$ tight inequalities) can therefore
uniquely be encoded by a bit vector $\uu\in\{0,1\}^d$ where $u_j=0$
means that the lower bound is tight in the $j$th pair of inequalities,
while $u_j=1$ means that the upper bound is tight. In fact, every bit
vector $\uu$ induces a vertex $\xx(\uu)$ defined by selecting from each
pair of inequalities the tight one according to $\uu$.

\begin{definition} For $\uu\in\{0,1\}^d$, we ley $\xx(\uu)\in\R^d$
be the vector recursively defined by
\begin{equation}\label{def:xu}
x_j(\uu) := (1-u_j)\varepsilon x_{j-1}(\uu) + u_j
(1-\varepsilon x_{j-1}(\uu)) = u_j + (1-2u_j)\varepsilon x_{j-1}(\uu),
\end{equation}
for $j=1,\ldots,d$, 
where we use the convention that $x_0(\uu)=0$. In particular, $\xx(\uu)$ is
one of $2^d$ vertices of the Klee-Minty cube (\ref{eq:km}).
\end{definition}

\subsection{Edges and Edge Directions}
Two vertices $\uu,\uu'$ are \emph{neighbors} if and only
if their convex hull is an edge (having $d-1$ tight
inequalities). This in turn is the case if and only if $\uu'$ is of
the form $\uu\oplus\{\ell\}$ (the bit vector obtained from $\uu$ by
flipping the $\ell$-th component). A general result~\cite[Lemma
3.6]{z-lop-95} entails the following key fact.

\begin{fact}\label{fact:neighbors}
Let $\ee\in\R^d$, $\uu\in\{0,1\}^d$. The following two statements are equivalent.
\begin{itemize}
\item[(i)] $\ee^T\xx(\uu) > \ee^T\xx$ for all $\xx$ in (\ref{eq:km})
\item[(ii)] $\ee^T\xx(\uu) > \ee^T\xx(\uu\oplus\{\ell\})$ for all
  $\ell\in\{1,2,\ldots,d\}$.
\end{itemize}
\end{fact}

Below we will use this fact together with Lemma~\ref{lem:shadow} to
prove that $\xx(\uu)$ yields a shadow vertex for all $\uu$. In order
to arrive at vectors $\cc,\dd$ that define a suitable shadow, we need
a little more notation.

\begin{definition} For $\uu\in\{0,1\}^d$ and
  $\ell\in\{1,2,\ldots,d\}$, we define
\begin{equation}\label{def:q}
q^{(\ell)}(\uu) = x_{\ell}(\uu\oplus\{\ell\}) - x_{\ell}(\uu), \quad \ell=1,2,\ldots,d.
\end{equation}
Furthermore, for $i,j\in\{1,2,\ldots,d\}$, we set
\begin{equation}\label{def:p}
p_i^j(\uu)=\prod_{k=i}^j(1-2u_k) \in\{-1,1\}.
\end{equation}
\end{definition}
\noindent Note that $p_i^j(\uu)$ simply encodes the parity of the bit
vector $u_i,u_{i+1},\ldots,u_j$. In order to apply
Fact~\ref{fact:neighbors}, we need to compute the edge directions.

\begin{lemma} Let $\uu\in\{0,1\}^d$ and $\ell\in\{1,2,\ldots,d\}$.
 Then
  \[
  x_j(\uu\oplus\{\ell\}) - x_j(\uu) =\left\{
  \begin{array}{ll}
    0, & \mbox{if $j<\ell$}, \\
     p_{\ell+1}^j(\uu)\varepsilon^{j-\ell} q^{(\ell)}(\uu)& \mbox{if
      $j\geq \ell$}.
  \end{array}\right.
\]

\end{lemma}

\begin{proof} By induction on $j$. For $j<\ell$, the values 
  $x_j(\uu\oplus\{\ell\})$ and $x_j(\uu)$ agree, since by (\ref{def:xu}), they
  only depend on bits $u_i,i\leq j<\ell$. For $j=\ell$, we recover
  (\ref{def:q}). For $j>\ell$, we use (\ref{def:xu}) to compute
\[
  x_{j}(\uu\oplus\{\ell\})-x_{j}(\uu)=
  (1-2u_j)\cdot\varepsilon  (x_{j-1}(\uu\oplus\{\ell\})-x_{j-1}(\uu)),
\]
and the statement follows from the inductive hypothesis.
\end{proof}

The previous lemma shows that all components of 
$\xx(\uu\oplus\{\ell\}))-\xx(\uu)$ are multiples
of $q^{(\ell)}(\uu)$, and it will be convenient to take out this factor.
\begin{definition}
For $\uu\in\{0,1\}^d$ and $\ell\in\{1,2,\ldots,d\}$, let $\yy^{(\ell)}(\uu)$ be
the vector defined by 
\begin{equation}\label{def:y}
 y^{(\ell)}_j(\uu)=\left\{
  \begin{array}{ll}
    0, & \mbox{if $j<\ell$}, \\
     p_{\ell+1}^j(\uu)\varepsilon^{j-\ell}& \mbox{if $j\geq \ell$}.
  \end{array}\right.
\end{equation}
We thus have 
\begin{equation}\label{eq:xy}
\xx(\uu\oplus\{\ell\})-\xx(\uu)   = \yy^{(\ell)}(\uu) \cdot q^{(\ell)}(\uu).
\end{equation}
\end{definition}

\subsection{The Shadow}
Let $\PP$ be the Klee-Minty cube as defined in (\ref{eq:km}). We want
to construct vectors $\cc$ and $\dd$ such that the shadow
$\pi_{\cc,\dd}(\PP)$ has the maximum of $2^d$ vertices. Our approach
is as follows. With a suitable $\cc$, we use
$\dd=(0,\ldots,0,1)$. For every $\uu\in\{0,1\}^d$,
we find a multiple $\dd(\uu)$ of $\dd$ such  that the vertex
$\xx(\uu)$ is the unique maximizer of the linear function $\ee(\uu)^T\xx :=
(\cc+\dd(\uu))^T\xx$ over (\ref{eq:km}).  With
Lemma~\ref{lem:shadow}, we conclude that $\pi_{\cc,\dd}(\xx(\uu))$ is
a shadow vertex.

\begin{definition}
For $\uu\in\{0,1\}^d$, let 
\[\cc :=
(\varepsilon^{3(d-1)},\varepsilon^{3(d-2)},\ldots,\varepsilon^3,0)\in\R^d\] and
\[\dd(u) :=
(0,0,\ldots,0,-\sum_{j=0}^{d-1}p_{j+1}^d(\uu)\varepsilon^{2(d-j)}) \in\R^d.
\]
\end{definition}

\begin{lemma}\label{lem:main}
Let $\uu\in\{0,1\}^d, \ell\in\{1,2\ldots,d\}, \ee(\uu) :=
\cc+\dd(u)$. For $\varepsilon<1/2$,
\[
\ee(\uu)^T(\xx(\uu\oplus\{\ell\})-\xx(\uu)) < 0,
\]
meaning that $\xx(\uu)$ has larger $\ee(\uu)^T\xx$-value than
all its neighbors. 
\end{lemma}
According to Fact~\ref{fact:neighbors}, $\xx(\uu)$ then uniquely maximizes 
$\ee(\uu)^T\xx$ over the Klee-Minty cube (\ref{eq:km}) and thus contributes to
the shadow by Lemma~\ref{lem:shadow}. It only remains to prove Lemma~\ref{lem:main}.

\begin{proof}
Making use of (\ref{eq:xy}), we first compute 
$\ee(\uu)^T\yy(u) = \cc^T\yy(\uu)+\dd(\uu)^T\yy(\uu)$. We
have
\begin{equation}\label{eq:cc}
\cc^T\yy(\uu) =
\sum_{j=\ell}^{d-1}\varepsilon^{3(d-j)}p_{\ell+1}^j(\uu)\varepsilon^{j-\ell}
=
\sum_{j=\ell}^{d-1}p_{\ell+1}^j(\uu) \varepsilon^{(3d-\ell)-2j}
\end{equation}
and
\begin{eqnarray}
\dd(\uu)^T\yy(\uu) &=& -
\displaystyle{\sum_{j=0}^{d-1}p_{j+1}^d(\uu)\varepsilon^{2(d-j)}p_{\ell+1}^d(\uu)\varepsilon^{d-\ell}}\nonumber\\ 
&=& - \displaystyle{\sum_{j=0}^{d-1}p_{j+1}^d(\uu) p_{\ell+1}^d(\uu)
  \varepsilon^{(3d-\ell)-2j}}.
\label{eq:dd}
\end{eqnarray}

For $j\geq\ell$, (\ref{def:p}) and the subsequent parity
interpretation of $p$ yields 
\[p_{j+1}^d(\uu)p_{\ell+1}^d(\uu)= p_{\ell+1}^j(\uu),\]
meaning that the terms for $j=\ell,\ldots,d-1$ in (\ref{eq:cc}) and
(\ref{eq:dd}) cancel, and we get
\[
\ee(\uu)^T\yy(\uu) = \cc^T\yy(\uu)+\dd(\uu)^T\yy(\uu)=- \sum_{j=0}^{\ell-1}p_{j+1}^d(\uu) p_{\ell+1}^d(\uu)
  \varepsilon^{(3d-\ell)-2j}.
\]
This expression is a polynomial in $\varepsilon$ whose nonzero
coefficients are in $\{-1,1\}$. Hence, for $\varepsilon<1/2$, the sign
of this polynomial is determined by the coefficient for $j=\ell-1$
which is
\[
-p_{\ell}^d(\uu)p_{\ell+1}^d(\uu) = -p_{\ell}^{\ell}(\uu) = -(1-2u_{\ell}).
\]

Now using (\ref{eq:xy}), our actual expression of interest
$\ee(\uu)^T(\xx(\uu\oplus\{\ell\})-\xx(\uu))$ has the same sign as
\begin{eqnarray*}
-(1-2u_{\ell})q^{(\ell)}(\uu) &=&
-(1-2u_{\ell})(x_{\ell}(\uu\oplus\{\ell\}) - x_{\ell}(\uu)) \\
&\stackrel{(\ref{def:xu})}{=}& -(1-2u_{\ell})^2 (1-2\varepsilon x_{\ell-1}(\uu)).
\end{eqnarray*}
By $\varepsilon<1/2$ and $x_{\ell-1}(\uu)\leq 1$, the sign of this
expression is negative, as desired. 
\end{proof}

\end{document}